\documentclass{amsart} 
\usepackage{amssymb}
\usepackage{mathrsfs}
\usepackage{cite}
\usepackage{graphicx}
\usepackage{caption}
\usepackage{subcaption}
\usepackage{float}
\usepackage{color}

\newtheorem{thm}{Theorem}[section]

\newtheorem{lem}[thm]{Lemma}

\theoremstyle{definition}

\theoremstyle{remark}
\newtheorem{remark}[thm]{Remark}

\numberwithin{equation}{section}

\DeclareMathOperator{\Ints}{\mathbb{Z}}  
\DeclareMathOperator{\Real}{\mathbb{R}}  
\DeclareMathOperator{\Torus}{\mathbb{T}}  
\DeclareMathOperator{\e}{e}                  
\DeclareMathOperator{\dtau}{d\it{\tau}}  
\DeclareMathOperator{\dw}{d\it{w}}                 
\DeclareMathOperator{\dx}{d\it{x}}                 
\DeclareMathOperator{\dy}{d\it{y}}                 

\DeclareMathOperator{\la}{\lambda}

\DeclareMathOperator{\ddt}{\frac{d}{d\textit{t}}} 
\DeclareMathOperator{\Tr}{\mathrm{Tr}} 
\DeclareMathOperator{\FTinv}{\mathcal{F}^{-1}} 

\begin{document}

\title[Lower Bounds for Sobolev Norm Blow-up]{Lower Bounds of Potential Blow-Up Solutions of the Three-dimensional Navier-Stokes Equations in $\dot{H}^\frac{3}{2}$}
\author{Alexey Cheskidov and Karen Zaya}
\address{Department of Mathematics, Statistics, and Mathematical Computer Science. University of Illinois at Chicago. 851 South Morgan Street (M/C 249). Chicago IL, 60607}
\email{acheskid@uic.edu, kzaya2@uic.edu}
\date{\today}

\begin{abstract}
We improve previous known lower bounds for Sobolev norms of potential blow-up solutions to the three-dimensional Navier-Stokes equations in $\dot{H}^{3/2}$.
\end{abstract}

\maketitle


\section{Introduction}

We consider the three-dimensional incompressible Navier-Stokes equations 
\begin{align} 
    \label{eq:NSE}
    &\frac{\partial u}{\partial t} + (u \cdot \nabla)  u = -\nabla p  + \nu \Delta u,\notag \\
    &\nabla \cdot u = 0, \\
    &u(x,0) = u_0 (x) \notag ,
\end{align}
where the velocity $u(x,t)$ and the pressure $p(x,t)$ are unknowns, $\nu > 0$ is the kinematic viscosity coefficient, the initial data $u_0 (x) \in L^2(\Omega)$, and the spatial domain $\Omega$ may have periodic boundary conditions or $\Omega = \Real^3$. The question of the regularity of solutions to \eqref{eq:NSE} remains open and is one of the Clay Mathematics Institute Millennium Prize problems. 

	In 1934, Leray \cite{Leray}  published his formative work on the the fluid equations. He proved the existence of global weak solutions to \eqref{eq:NSE} and proved that smooth solutions are unique in the class of Leray-Hopf solutions. He also showed that if $\|u(t)\|_{H^1}$ is continuous on $[0,T^*)$ and blows up at time $T^*$, then
\begin{align}
    \label{eq:LerayBound}
    \|u(t)\|_{\dot{H}^1(\Real^3)} \ge \frac{c }{(T^* - t)^\frac{1}{4}}.
\end{align}
	
	Moreover, the bound for $L^p$ norms for $3< p < \infty$, 
\begin{align}
    \label{LpBlowUpBound}
    \| u(t) \|_{L^p(\Real^3)} \ge \frac{c_p}{(T^*-t)^\frac{p-3}{2p}},
\end{align}
have been known for a long time (see \cite{Leray} and \cite{Giga}). The Sobolev embedding $\dot{H}^s(\Real^3) \subset L^\frac{6}{3-2s}(\Real^3)$ and \eqref{LpBlowUpBound} yield that	
\begin{align}
    \label{eq:GigaBound}
    \|u(t)\|_{\dot{H}^s(\Omega)} \ge \frac{c}{(T^* - t)^\frac{2s-1}{4}},
\end{align}
for $\frac{1}{2} < s < \frac{3}{2}$ and $\Omega = \Real^3$. Robinson, Sadowski, and Silva extended \eqref{eq:GigaBound} in \cite{RobinsonSadSilva} for $\frac{3}{2} < s < \frac{5}{2}$ for the whole space and in the presence of periodic boundary conditions. This bound is considered optimal for those values of $s$. 
	
	When $s > \frac{5}{2}$, Benameur  \cite{Benameur} showed
\begin{align*}
    \| u(t) \|_{\dot{H}^s(\Real^3)} \ge \frac{c(s) \|u(t)\|_{L^2(\Real^3)}^\frac{3-2s}{3} }{(T^* - t)^\frac{s}{3}},
\end{align*}
which was improved upon by Robinson, Sadowski, and Silva in \cite{RobinsonSadSilva} to
\begin{align}
    \label{eq:RobSadSilBound52}
    \| u(t) \|_{\dot{H}^s(\Omega)} \ge \frac{c(s) \|u_0\|_{L^2(\Omega)}^\frac{5-2s}{5} }{(T^* - t)^\frac{2s}{5}},
\end{align}
when $\Omega = \Torus^3$ or $\Omega = \Real^3$. 

The border cases $s = \frac{3}{2}$ and $s = \frac{5}{2}$ required separate treatment. For $s = \frac{3}{2}$, Robinson, Sadowski, and Silva had an epsilon correction for the case with periodic boundary conditions. In \cite{CortissozMonPin}, Cortissoz, Montero, and Pinilla improved the bound for $s = \frac{3}{2}$ on $\Torus^3$, but they had a logarithmic correction:
\begin{align}
    \label{eq:CMP1}
    \| u(t) \|_{\dot{H}^\frac{3}{2}(\Torus^3)} \ge \frac{c}{ \sqrt{(T^*-t) |\log(T^* - t)|}  }.
\end{align}

For $s = \frac{5}{2}$, Cortissoz, Montero, and Pinilla \cite{CortissozMonPin} also found 
\begin{align}
    \label{eq:CMP2}
    \| u(t) \|_{\dot{H}^\frac{5}{2}(\Omega)} \ge \frac{c}{(T^*-t)  | \log(T^* - t)| },
\end{align}
when $\Omega = \Torus^3$ or $\Omega = \Real^3$. In \cite{McCormickEtAl}, the authors proved
\begin{align}
     \limsup_{t \rightarrow T^{*-}} (T^*-t) \|u(t)\|_{\dot{H}^{5/2}(\Omega)} \ge c. 
\end{align}

	In this paper, we improve the bound for the $\dot{H}^\frac{3}{2}(\Omega)$-norm to the optimal bound \eqref{eq:GigaBound} when $\Omega = \Real^3$ or $\Omega = \Torus^3$. Our method is not contingent on rescaling arguments and thus works simultaneously for $\Real^3$ and $\Torus^3$. We stress the importance of the $H^{3/2}$ norm, which scales to the $L^\infty$ norm and corresponds to the uncovered limit of \eqref{eq:GigaBound}.  We also note the significance of $H^{5/2}$, which is a critical space for the Euler equations and scales like $B^1_{\infty, \infty}$, the Beale-Kato-Majda space. Furthermore, the persistence of the logarithmic correction in estimate \eqref{eq:CMP2} is consistent with the recent result of Bourgain and Li \cite{BourgainLi} on the ill-posedness of the Euler equations in $H^{5/2}$.
	
\begin{remark}
    The lower bound for the $\dot{H}^\frac{3}{2}$-norm of blow-up solutions was also presented in papers by Montero \cite{Montero} and McCormick, Olson, Robinson, Rodrigo, Vidal-Lopez, and Zhou \cite{McCormickEtAl}, which both appeared shortly after this paper. 
\end{remark}
	
	Our methods differ from previous works as we utilize Littlewood-Paley decomposition of solutions $u$ of \eqref{eq:NSE} for much of the paper. We denote wave numbers as $\la_q = 2^q$ (in some wave units). For $\psi \in C^\infty(\Omega)$, we define 
\begin{displaymath}
   \psi(\xi) = \left\{
     \begin{array}{lr}
       1 ~~:& | \xi | \le \frac{1}{2} \\
       0 ~~:& | \xi | > 1.
     \end{array}
   \right.
\end{displaymath} 
Next define $\phi(\xi) = \psi(\xi / \la_1) - \psi(\xi)$ and $\phi_q(\xi) = \phi(\xi / \la_q)$. Then 
\begin{align}
    \label{eq:LPDecomp}
    u = \sum_{q = -\infty}^\infty u_q,
\end{align}
in the sense of distributions, where the $u_q$ is the $q^{th}$ Littlewood-Paley piece of $u$. On $\Real^3$, the Littlewood-Paley pieces are defined as 
\begin{align}
    u_q (x) =  \int_{\Real^3} u(x-y) \mathcal{F}^{-1}(\phi_q)(y) \dy,
\end{align}
where $\mathcal{F}$ is the Fourier transform. In the periodic case, the Littlewood-Paley pieces are given by
\begin{align}
    u_q (x) =  \sum_{k \in \Ints^3} \hat{u}(k) \phi_q(k) \e^{i k \cdot x},
\end{align}
where \eqref{eq:LPDecomp} holds provided $u$ has zero-mean. Moreover, $u_q = 0$ in the periodic case when $q<0$. We will use the notation
\begin{align*}
    u_{\le Q} = \sum_{q \le Q} u_q,  ~ u_{\ge Q} = \sum_{q \ge Q} u_q.
\end{align*}
We define the homogeneous Sobolev norm of $u$ as
\begin{align}
    \label{eq:SobNormDef}
    \| u\|_{\dot{H}^s} = \Big( \sum_{q = -\infty}^\infty \la_q^{2s} \|u_q\|_{L^2}^2  \Big)^\frac{1}{2}.
\end{align}
Note that it corresponds to the nonhomogeneous Sobolev norm $H^s$ in the periodic case.

	We suppress $L^p$ norm notation as  $\| \cdot \|_p :=  \| \cdot \|_{L^p} $. We will also suppress the notation for domains for integrals and functional spaces, i.e. $\int := \int_\Omega$. All $L^p$ and Sobolev spaces are over $\Omega,$ where $\Omega$ either has periodic boundary conditions or is the whole space $\Real^3$, as described in the introduction (unless explicitly otherwise stated). The methods of proof apply to either domain. Sobolev spaces are denoted by $H^s$ and homogeneous Sobolev spaces by $\dot{H}^s$.  We will use the symbol $\lesssim$ (or $\gtrsim$) to denote that an inequality that holds up to an absolute constant. 

\section{Bounding Blow-Up for $ s = \frac{3}{2}$}

We begin by testing the weak formulation of the Navier-Stokes equation with $\la_q^{2s}  (u_q)_q$ to obtain
\begin{align}
    \label{eq:L2Tested}
    \ddt \big( \la_q^{2s} \|u_q\|^2_2 \big) = -\nu \la_q^{2s+2} \|u_q\|^2_2  + 2 \la_q^{2s}  \int \Tr [( u \otimes u)_q \cdot \nabla u_q] ~\dx. 
\end{align}
 In the typical fashion, we write 
\begin{align}
     \label{eq:Tensor}
     (u \otimes u)_q = u_q \otimes u + u \otimes u_q + r_q(u,u),
\end{align}
for $q > -1$, where the remainder function is given by 
\begin{align}
     \label{eq:Remainder}
     r_q(u,u)(x) = \int \FTinv(\phi_q)(y) (u(x-y)-u(x)) \otimes (u(x-y)-u(x)) ~\dy.
\end{align}
Thus, we rewrite the nonlinear term as 
\begin{align}
    \label{eq:NLTRemainer}
     \int \Tr [( u \otimes u)_q \cdot \nabla u_q] \dx 
     = \int r_q(u,u)  \cdot \nabla u_q \dx -  \int u_q \cdot \nabla u_{\le q+1} \cdot u_q \dx.
\end{align}

\begin{lem}
\label{NLTLemma}
The integral \eqref{eq:NLTRemainer} corresponding to the nonlinear term in \eqref{eq:L2Tested} is bounded above by
\begin{align}
\begin{aligned}
    \label{eq:TraceBound}
     \int \Tr [( u \otimes u)_q \cdot \nabla u_q] ~\dx 
     \lesssim  & \la_q^{-1} \|u_q\|_2 \sum_{p=-\infty}^{q} \la_p^2 \|u_p\|_4^2 \\
      & +  \la_q \|u_q\|_2 \sum_{p=q+1}^\infty \|u_p\|_4^2 \\
      & + \|u_q\|_2^2 \sum_{p=-\infty}^{q+1} \la_p^\frac{5}{2} \|u_p\|_2 .
\end{aligned}
\end{align}
\end{lem}
\begin{proof} We examine the two integrals on the right-hand side of  \eqref{eq:NLTRemainer} separately.  By H\"older's inequality, 
\begin{align*}
     \int r_q(u,u)  \cdot \nabla u_q \dx  \lesssim \|r_q(u,u)\|_2 \la_q \|u_q\|_2.
\end{align*}
We use Littlewood-Paley decomposition and split the sum into low versus high modes to find
\begin{align*}
     \|r_q(u,u)\|_2 
      \lesssim &  \int_{\Real^3} | \FTinv(\phi_q)(y) | ~\| u(x-y)-u(x) \|_4^2 \dy \\
      \lesssim & \int_{\Real^3} | \FTinv(\phi_q)(y) | \sum_{p = -\infty}^q \| (u(x-y)-u(x))_p  \|_4^2 \dy \\
     & +  \int_{\Real^3} | \FTinv(\phi_q)(y) | \sum_{p = q+1}^\infty \| (u(x-y)-u(x))_p  \|_4^2 \dy.
\end{align*}
We apply the Mean-Value Theorem on the low modes and the triangle inequality on the high modes to arrive at
\begin{align*}
    \|r_q(u,u)\|_2 
    \lesssim &  \int_{\Real^3} | \FTinv(\phi_q)(y) | ~ |y|^2 \sum_{p=-\infty}^q \| \nabla u_p \|_4^2  \dy \\
    & + \int_{\Real^3} | \FTinv(\phi_q)(y) | \sum_{p=q+1}^\infty \|u_p\|_4 \dy \\
    \lesssim &   \la_q^{-2} \sum_{p=-\infty}^q \la_p^2 \| u_p \|_4^2 
    + \sum_{p=q+1}^\infty \|u_p\|_4
\end{align*}
Thus, 
\begin{align}
\begin{aligned}
    \label{eq:GeneralRemainderNorm} 
     \int r_q(u,u)  \cdot \nabla u_q \dx 
     \lesssim & \la_q^{-1} \| u_q \|_2 \sum_{p=-\infty}^q \la_p^2 \| u_p \|_4^2 \\
     & + \la_q \|u_q \|_2 \sum_{p=q+1}^\infty \|u_p\|_4^2.
\end{aligned}
\end{align}
For the second term of \eqref{eq:NLTRemainer}, we use a similar process as above in addition to Bernstein's inequality to find
\begin{align}
    \label{eq:ThirdNLTEstimate}
    \int u_q \cdot \nabla u_{\le q+1} \cdot u_q \dx 
    & \lesssim \|u_q\|_2^2 \sum_{p=-\infty}^{q+1} \la_p \|u_p\|_\infty \\
    & \lesssim \|u_q\|_2^2 \sum_{p=-\infty}^{q+1} \la_p^{5/2} \|u_p\|_2 \notag.
\end{align}
Combining \eqref{eq:GeneralRemainderNorm} and \eqref{eq:ThirdNLTEstimate} yields the desired bound \eqref{eq:TraceBound}. 
\end{proof}

Similar estimates were executed in \cite{CCFSOnsager} and \cite{CheskidovShvydkoyRegularity}. We apply the bound obtained in Lemma \ref{NLTLemma} to write
\begin{align}
    \label{eq:Flux1}
    \ddt \sum_{q = -\infty}^\infty \big( \la_q^{2s} \|u_q\|_2^2 \big)
     \lesssim -\sum_{q = -\infty}^\infty \big( \nu \la_q^{2s + 2}  \|u_q\|_2^2 \big) + 2\big(A + B + C\big),
\end{align}
where
\begin{align}
    \label{eq:COVNLPiecesA}
    &A = \sum_{q = -\infty}^\infty \sum_{p=-\infty}^q \la_q^{2s-1} \|u_q\|_2 \la_p^2 \|u_p\|_4^2 , \\
    \label{eq:COVNLPiecesB}
    &B = \sum_{q = -\infty}^\infty  \sum_{p = q+1}^\infty  \la_q^{2s+1} \|u_q\|_2 \|u_p\|_4^2, \\
    \label{eq:COVNLPiecesC}
    &C = \sum_{q = -\infty}^\infty \sum_{p=-\infty}^{q+1} \la_q^{2s} \|u_q\|_2^2   \la_p^{5/2} \|u_p\|_2. 
 \end{align}

\begin{thm} 
    \label{thm:Ricatti1}
Let $u$ be a solution to \eqref{eq:NSE} with finite energy initial data. Then for $s = \frac{3}{2}$, the solution $u$ satisfies the Riccati-type differential inequality
\begin{align}
    \label{eq:Ricatti32}
    \ddt \sum_{q = -\infty}^\infty \Big( \la_q^{3} \|u_q\|_2^2 \Big)
     \lesssim \sum_{q = -\infty}^\infty \Big( \la_q^{3}  \|u_q\|_2^2 \Big)^2
\end{align}
\end{thm}

\begin{proof}
We bound the nonlinear terms.  First, we estimate \eqref{eq:COVNLPiecesA} for $s = \frac{3}{2}$. We apply Bernstein's inequality in three-dimensions and we rewrite the sum
\begin{align*}
    A &= \sum_{q = -\infty}^\infty \sum_{p=-\infty}^q \la_q^2 \|u_q\|_2 \la_p^2 \|u_p\|_4^2\\
    & \lesssim \sum_{q = -\infty}^\infty \sum_{p=-\infty}^q \la_q^2 \|u_q\|_2 \la_p^{7/2} \|u_p\|_2^2 \\
    & = \sum_{q = -\infty}^\infty \sum_{p=-\infty}^q \la_{q-p}^{-1/2} \big(\la_q^{5/2} \|u_q\|_2 \big) \big( \la_p^3 \|u_p\|^2_2 \big).
\end{align*}
We apply the Cauchy-Schwartz inequality to yield
\begin{align}
    \label{eq:YoungA}
    A \lesssim \sum_{q = -\infty}^\infty \sum_{p=-\infty}^q  \la_{q-p}^{-1/2} \Big(\frac{\nu}{3} \la_q^5 \|u_q\|_2^2\Big) 
    +  \la_{q-p}^{-1/2} \Big(\nu^{-1} \la_p^3 \|u_p\|_2^2 \Big)^2.
\end{align}
Next we sum in $p$ for the first term and exchange the order of summation and sum in $q$ for the second term of \eqref{eq:YoungA}:
\begin{align}
    \label{eq:EstimateA}
    A \lesssim \sum_{q = -\infty}^\infty  \Big( \nu^{-1} \la_q^3 \|u_q\|_2^2 \Big)^2  + \frac{\nu}{3} \sum_{q = -\infty}^\infty \Big( \la_q^5 \|u_q\|_2^2 \Big).
\end{align}

To estimate \eqref{eq:COVNLPiecesB} when $s = \frac{3}{2}$, first we apply Bernstein's inequality for three-dimensions to find\begin{align*}
    B & = \sum_{q = -\infty}^\infty  \sum_{p = q+1}^\infty  \la_q^4 \|u_q\|_2 \|u_p\|_4^2 \\
    & \lesssim \sum_{q = -\infty}^\infty  \sum_{p = q+1}^\infty  \la_q^4 \|u_q\|_2 \la_p^{3/2} \|u_p\|_2^2. 
\end{align*}
We rewrite the sum to look like 
\begin{align*}
     B \lesssim \sum_{q = -\infty}^\infty  \sum_{p = q+1}^\infty \la_{p-q}^{-5/2} \big( \la_q^{3/2} \|u_q\|_2 \big) \big( \la_p^{3/2} \|u_p\|_2 \big) \big( \la_p^{5/2} \|u_p\|_2 \big).
\end{align*}
We apply Young's inequality with the exponents $\theta_1 = \theta_2 = 4$ and $\theta_3 = 2$ to yield
\begin{align}
\begin{aligned}
    \label{eq:YoungB}
    B \lesssim & \sum_{q = -\infty}^\infty  \sum_{p = q+1}^\infty  \la_{p-q}^{-5/2} \Big( \nu^{-1} \la_q^3 \|u_q\|_2^2 \Big)^2   \\
    & +  \sum_{q = -\infty}^\infty  \sum_{p = q+1}^\infty  \la_{p-q}^{-5/2} \Big( \nu^{-1} \la_p^{3} \|u_p\|_2^2 \Big)^2 \\
    & + \sum_{q = -\infty}^\infty  \sum_{p = q+1}^\infty \la_{p-q}^{-5/2}  \Big( \frac{\nu}{3} \la_p^{5} \|u_p\|_2^2 \Big).
\end{aligned}
\end{align}
Next we sum in $p$ for the first term and exchange the order of summation and sum in $q$ for the second and third terms of \eqref{eq:YoungB}. Note the summation in $q$ converges:
\begin{align*}
    B \lesssim \sum_{q=-\infty}^{\infty} \Big(\nu^{-1} \la_q^3 \|u_q\|_2^2 \Big)^2
    + \sum_{p=-\infty}^\infty \Big[ \Big(\nu^{-1} \la_p^3 \|u_p\|_2^2 \Big)^2
    + \Big( \frac{\nu}{3} \la_p^5 \|u_p\|_2^2 \Big) \Big]. 
\end{align*}
Thus we arrive at the bound 
\begin{align}
    \label{eq:EstimateB}
    B \lesssim \sum_{q = -\infty}^\infty  \Big( \nu^{-1} \la_q^3 \|u_q\|_2^2 \Big)^2  + \frac{\nu}{3} \sum_{q = -\infty}^\infty \Big( \la_q^5 \|u_q\|_2^2 \Big).
\end{align}

	Finally, we estimate \eqref{eq:COVNLPiecesC} for $s = \frac{3}{2}$. We rewrite the sum 
\begin{align*}
    C &= \sum_{q = -\infty}^\infty \sum_{p=-\infty}^{q+1}  \la_q^3 \|u_q\|_2^2 \la_p^{5/2} \|u_p\|_2^2 \\
    &= \sum_{q = -\infty}^\infty \sum_{p=-\infty}^{q+1} \la_{q-p}^{-\delta} \big(\la_q^{3/2} \|u_q\|_2 \big)^{2-\delta} \big( \la_q^{5/2} \|u_q\|_2^2 \big)^\delta \big( \la_p^{3/2} \|u_p\|_2 \big)^\delta \big( \la_p^{5/2} \|u_p\|_2 \big)^{1-\delta},
\end{align*}
where $\delta$ is a small positive number we can choose. We apply Young's inequality with
\begin{align*}
    \theta_1 = \frac{4}{2-\delta}, ~~~
    \theta_2 = \frac{2}{\delta},~~~
    \theta_3 = \frac{4}{\delta}, ~~~
    \theta_4 = \frac{2}{1-\delta},
\end{align*}
where we require $\delta < 1$ to ensure the exponents are all positive and indeed $\frac{1}{\theta_1} + \frac{1}{\theta_2} + \frac{1}{\theta_3} + \frac{1}{\theta_4}=1$. Then we have 
\begin{align}
\begin{aligned}
    \label{eq:YoungC}
    C \lesssim & \sum_{q = -\infty}^\infty \sum_{p=-\infty}^{q+1} 
    \Big[ \la_{q-p}^{-\delta} \Big( \nu^{-1} \la_q^3 \|u_q\|_2^2 \Big)^2
    + \la_{q-p}^{-\delta} \Big( \frac{\nu}{6} \la_q^5 \|u_q\|_2^2 \Big) \Big] \\
    & + \sum_{q = -\infty}^\infty \sum_{p=-\infty}^{q+1}
    \Big[ \la_{q-p}^{-\delta}  \Big(\nu^{-1} \la_p^3 \|u_p\|_2^2 \Big)^2 
    + \la_{q-p}^{-\delta} \Big( \frac{\nu}{6} \la_p^5 \|u_p\|_2^2 \Big) \Big],
\end{aligned}
\end{align}
For the first two terms of \eqref{eq:YoungC}, we sum in $p$. For the third and fourth terms, we exchange the order of summation and sum in $q$ to arrive at
\begin{multline*}
    C \lesssim \sum_{q = -\infty}^\infty 
    \Big[ \Big(\nu^{-1} \la_q^3 \|u_q\|_2^2 \Big)^2
    + \Big(\frac{\nu}{6} \la_q^5 \|u_q\|_2^2 \Big) \Big] \\
    + \sum_{p=-\infty}^{\infty}
    \Big[ \Big( \nu^{-1} \la_p^3 \|u_p\|_2^2 \Big)^2
    + \Big( \frac{\nu}{6} \la_p^5 \|u_p\|_2^2 \Big) \Big].
\end{multline*}
Note $\delta$ positive ensures the summation in $q$ converges. Rewriting the above inequality yields 
\begin{align}
    \label{eq:EstimateC}
    C \lesssim \sum_{q = -\infty}^\infty  \Big( \nu^{-1} \la_q^3 \|u_q\|_2^2 \Big)^2  + \frac{\nu}{3} \sum_{q = -\infty}^\infty \Big( \la_q^5 \|u_q\|_2^2 \Big).
\end{align}

	We use the estimates \eqref{eq:EstimateA}, \eqref{eq:EstimateB}, and \eqref{eq:EstimateC} in \eqref{eq:Flux1} with $s= \frac{3}{2}$ to get the Ricatti-type differential inequality
\begin{align}
    \label{Ricatti1}
     \ddt \sum_{q = -\infty}^\infty \Big( \la_q^3 \|u_q\|_2^2 \Big)
     \lesssim \sum_{q = -\infty}^\infty \Big( \nu^{-1} \la_q^3  \|u_q\|_2^2 \Big)^2.
\end{align} 
\end{proof}

\begin{remark}
The method used to prove Theorem \ref{thm:Ricatti1} works for $\frac{1}{2} < s < \frac{5}{2}$. Instead of \eqref{eq:Ricatti32}, one must show 
\begin{align}
    \label{eq:Ricattis}
    \ddt \sum_{q = -\infty}^\infty \Big( \la_q^{2s} \|u_q\|_2^2 \Big)
     \lesssim \sum_{q = -\infty}^\infty \Big( \la_q^{2s}  \|u_q\|_2^2 \Big)^\frac{2s+1}{2s-1}.
\end{align}
In the proof for \eqref{eq:Ricattis}, one must treat the three cases $\frac{1}{2} < s < \frac{3}{2}$, $s = \frac{3}{2}$, and  $\frac{3}{2} < s < \frac{5}{2}$ separately, but in analogous manners.  
\end{remark}

\begin{thm}
\label{thm:BlowUpRate32}
Let $u$ be a smooth solution to \eqref{eq:NSE} with finite energy initial data such that $u$ loses regularity at time $T^*$. Then 
\begin{align}
    \label{eq:H32BlowUpRate}
    \|u (t) \|_{\dot{H}^{3/2} (\Omega)} \ge \frac{c}{\sqrt{T^*-t}} ,
\end{align}
for $0\le t < T^*$ and $\Omega = \Torus^3$ or $\Omega = \Real^3$.
\end{thm}
\begin{proof}
Let $y (t) = \| u (t) \|^2_{\dot{H}^{3/2}}$. By Theorem \ref{thm:Ricatti1}, $y$ satisfies the differential inequality 
\begin{align}
    \label{eq:RiccattiInY}
    \ddt y(t) \lesssim  y(t)^2.
\end{align}
Rearranging the inequality and integrating from time $t$ to blow-up time $T^*$ yields
\begin{align*}
    \int_{y(t)}^\infty \frac{\dw}{w^2} \lesssim \int_t^{T^*} \dtau,
\end{align*}
which becomes 
\begin{align*}
    \frac{1}{y(t)} \lesssim T^* -t.
\end{align*}
Then, as desired
\begin{align}
    \|u(t)\|_{\dot{H}^{3/2}(\Omega)} \ge \frac{c}{\sqrt{T^* - t}},
\end{align} 
for $0\le t < T^*$ and $\Omega = \Torus^3$ or $\Omega = \Real^3$.
\end{proof}

\begin{remark}
    The procedure in Theorem \ref{thm:BlowUpRate32} can be applied to  \eqref{eq:Ricattis} for $y(t) = \| u(t) \|^2_{\dot{H}^{s}(\Omega)}$ to yield 
\begin{align}
    \|u(t)\|_{\dot{H}^s(\Omega)} \ge \frac{c}{(T^* - t)^\frac{2s-1}{4}},
\end{align}
for $\frac{1}{2} < s < \frac{5}{2}$, $0\le t < T^*$, and $\Omega = \Real^3$ or $\Omega = \Torus^3$.
\end{remark}


\textbf{Acknowledgments:}  The work of A. Cheskidov was partially supported by NSF Grant DMS-1108864. The work of K. Zaya was partially supported by NSF grant DMS-1210896. Part of the work was carried out while K. Zaya was visiting and received some support from the Institute for Pure \& Applied Mathematics, an NSF Math Institute at UCLA, during the Mathematics of Turbulence long program.

\bibliographystyle{plain}
\bibliography{SobolevRevised}

\end{document}